\documentclass[11pt]{amsart}
\usepackage[margin=1in]{geometry}

\usepackage{amssymb}
\usepackage{amsthm}
\usepackage{amsmath}
\usepackage{mathrsfs}
\usepackage{amsbsy}
\usepackage[all]{xy}
\usepackage{bm}
\usepackage{hyperref}
\usepackage{tikz}
\usetikzlibrary{patterns}
\usepackage{array}
\usepackage{float}
\usepackage{enumerate}
\usepackage{xcolor}
\usepackage{hhline}
\allowdisplaybreaks
\usepackage[noadjust]{cite}

\usepackage{caption}
\usepackage{tabu}
\usepackage{diagbox}

\newcommand{\cat}{\mathsf{CAT}}
\newcommand{\C}{\mathsf{C}}
\newcommand{\A}{\mathsf{A}}
\newcommand{\T}{\mathsf{T}}

\DeclareMathOperator{\Id}{Id}

\usepackage[noabbrev,capitalise]{cleveref}

\newenvironment{enumerate*}%
  {\begin{enumerate}[(I)]%
    \setlength{\itemsep}{10pt}%
    \setlength{\parskip}{0pt}}%
  {\end{enumerate}}

\newtheorem{theorem}{Theorem}

\newtheorem{lemma}[theorem]{Lemma}

\theoremstyle{definition}

\begin{document}

\title[]{Cats in cubes}

\author{Noga Alon}
\address{Department of Mathematics, Princeton University, Princeton, NJ 08544, USA and Schools of Mathematics and Computer Science, Tel Aviv University, Tel Aviv 69978, Israel}
\email{nalon@math.princeton.edu}
\author{Noah Kravitz}
\address[]{Department of Mathematics, Princeton University, Princeton, NJ 08544, USA}
\email{nkravitz@princeton.edu}

\maketitle

\begin{abstract}
Answering a recent question of Patchell and Spiro, we show that when a $d$-dimensional cube of side length $n$ is filled with letters, the word $\mathsf{CAT}$ can appear contiguously at most $(3^{d-1}/2)n^d$ times (allowing diagonals); we also characterize when equality occurs and extend our results to words other than $\mathsf{CAT}$.
\end{abstract}

\section{Cats in grids and cubes}
\subsection{Cats in grids}
Suppose we label an $n \times n$ grid with the letters $\C$, $\A$, and $\T$.  We can look for triples of consecutive letters (appearing horizontally, vertically, and diagonally as in a wordsearch) that spell out the word $\cat$.  What is the maximum number of such $\cat$'s in the grid?  Patchell and Spiro \cite{sam} showed that the number of $\cat$'s is at most $2n^2$.  They also constructed an example with $(3/2)n^2-O(n)$ $\cat$'s and conjectured that this construction is asymptotically optimal.\footnote{The formulation in terms of $\cat$'s appears in Sam Spiro's personal list of open problems \cite{sam2}.}  In this paper, we will confirm their conjecture in a very precise sense.

It is more natural to replace the $n \times n$ grid with the discrete torus $(\mathbb{Z}/n\mathbb{Z})^2$.  Notice that this operation of ``gluing'' opposite sides of the $n \times n$ grid can add only $O(n)$ $\cat$'s.  Patchell and Spiro's construction is as follows, for $n$ a multiple of $4$: Label the point $(x,y)$ with the letter $\A$ if $x$ is even; label $(x,y)$ with $\C$ if $x \equiv 1 \pmod{4}$; and label $(x,y)$ with $\T$ if $x \equiv 3 \pmod{4}$.  In other words, the letters appear in vertical stripes, and the order of the stripes is $$\ldots \C\A\T\A\C\A\T\A\ldots.$$
See the left part of Figure~\ref{fig:examples}.  Let us count the number of $\cat$'s in this labeling: Each $\C$ is part of exactly $6$ $\cat$'s ($2$ horizontal $\cat$'s and $4$ diagonal $\cat$'s), and there are $n^2/4$ $\C$'s, so we get $(3/2)n^2$ $\cat$'s in total.

\begin{figure}[htb]
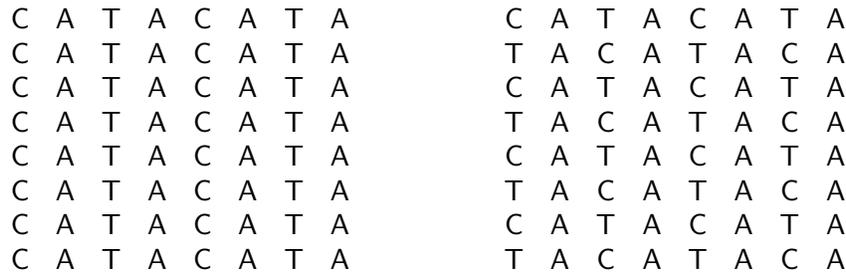

    \centering
    $\begin{matrix}
\C & \A & \T & \A & \C & \A & \T & \A \\
\C & \A & \T & \A & \C & \A & \T & \A \\
\C & \A & \T & \A & \C & \A & \T & \A \\
\C & \A & \T & \A & \C & \A & \T & \A \\
\C & \A & \T & \A & \C & \A & \T & \A \\
\C & \A & \T & \A & \C & \A & \T & \A \\
\C & \A & \T & \A & \C & \A & \T & \A \\
\C & \A & \T & \A & \C & \A & \T & \A
    \end{matrix}$ \hspace{1.8cm}
    $\begin{matrix}
\C & \A & \T & \A & \C & \A & \T & \A \\
\T & \A & \C & \A & \T & \A & \C & \A \\
\C & \A & \T & \A & \C & \A & \T & \A \\
\T & \A & \C & \A & \T & \A & \C & \A \\
\C & \A & \T & \A & \C & \A & \T & \A \\
\T & \A & \C & \A & \T & \A & \C & \A \\
\C & \A & \T & \A & \C & \A & \T & \A \\
\T & \A & \C & \A & \T & \A & \C & \A
    \end{matrix}$
    \caption{The original (left) and modified (right) $2$-dimensional Patchell--Spiro labelings for $n=8$, drawn without wrap-around.  Notice that each $\A$ is part of $3$ $\cat$'s and each $\C$ or $\T$ is part of $6$ $\cat$'s.}
    \label{fig:examples}
\end{figure}

We will also need the following modification of the Patchell--Spiro construction: Label $(x,y)$ with $\A$ if $x$ is even; label $(x,y)$ with $\C$ if $x$ is odd and $x+2y \equiv 1 \pmod{4}$; and label $(x,y)$ with $\T$ if $x$ is odd and $x+2y \equiv 3 \pmod{4}$.  In other words, the $\A$'s appear in vertical stripes as before, and the other vertical stripes alternately have the patterns
$$\ldots \T\C\T\C\T\C \cdots \quad \text{and} \quad \ldots \C\T\C\T\C\T \ldots .$$
See the right part of Figure~\ref{fig:examples}.  By the same computation as in the previous paragraph, there are $(3/2)n^2$ $\cat$'s.  Our main result (below) will imply that these labelings contain the maximum possible number of cats.

\subsection{Cats in cubes}
Spiro and Patchell also asked what happens in higher dimensions.  Now, we label $(\mathbb{Z}/n\mathbb{Z})^d$ ($n \geq 3$) with the letters $\C$, $\A$, and $\T$.  An occurrence of $\cat$ is a triple of points $x, x+y, x+2y$ with the labels $\C$, $\A$, $\T$ (in that order), where $x \in (\mathbb{Z}/n\mathbb{Z})^d$ and $y \in \{-1,0,1\}^d \setminus \{0\}$.
 
As expected, we have higher-dimensional analogues of the constructions from the previous subsection.  Suppose $n$ is a multiple of $4$.  Fix some $\epsilon_y \in \{1,3\}$ for each $y=(y_1,\ldots, y_{d-1}) \in \{0,1\}^{d-1}$, and  label the point $(x_1, \ldots, x_d) \in (\mathbb{Z}/n\mathbb{Z})^d$ as follows.  If $x_d$ is even, then label $(x_1, \ldots, x_d)$ with $\A$.  Now suppose $x_d$ is odd, and let $y\in \{0,1\}^{d-1}$ be the reduction  of $(x_1, \ldots, x_{d-1})$ modulo $2$; label $(x_1, \ldots, x_d)$ with $\C$ if $x_d \equiv \epsilon_y \pmod{4}$, and label it with $\T$ if $x_d \equiv \epsilon_y+2 \pmod{4}$.
In other words, we label each coset $(y,1)+(2\mathbb{Z}/n\mathbb{Z})^d$ with alternating hyperplanes (orthogonal to the $x_d$-direction) of $\C$'s and $\T$'s, and $\epsilon_y$ determines whether we ``start'' with $\C$ or with $\T$.  We can  obtain equivalent constructions by translating and by distinguishing a coordinate other than $x_d$, and we obtain $d\cdot 2^{2^{d-1}+1}$ examples in this way.  The same computation as in the previous subsection shows that these $d\cdot 2^{2^{d-1}+1}$ \emph{generalized Patchell--Spiro labelings} each have $(3^{d-1}/2)n^d$ $\cat$'s.  Our main result is that these labelings are optimal in all dimensions.

\begin{theorem}\label{thm:multi}
In any labeling of $(\mathbb{Z}/n\mathbb{Z})^d$ ($n \geq 3$) with the letters $\C$, $\A$, and $\T$, the number of $\cat$'s is at most $(3^{d-1}/2)n^d$.  Moreover, equality is attained if and only if $n$ is a multiple of $4$ and the labeling is a generalized Patchell--Spiro labeling.
\end{theorem}

We remark that the analogous result holds more generally in $(\mathbb{Z}/n_1\mathbb{Z}) \times \cdots \times (\mathbb{Z}/n_d\mathbb{Z})$: 
The maximum number of $\cat$'s is at most $(3^{d-1}/2)n_1\cdots n_d$, and equality is attained exactly for the generalized Patchell--Spiro labelings (which exist when some $n_i$ is a multiple of $4$ and all $n_i$'s are at least $3$).  Since the proof of this asymmetric version is identical to the proof of the symmetric version, we stick with the latter for ease of exposition.

Our proof of Theorem~\ref{thm:multi} uses simple tools from spectral graph theory (or, equivalently, Fourier analysis).  We remark that the $2$-dimensional case of this theorem can also be proven in a completely elementary way but that this hands-on approach seems not to generalize to higher dimensions.

\subsection{Organization}
In Section~\ref{sec:higher}, we prove Theorem~\ref{thm:multi}.  In Section~\ref{sec:conclusion}, we briefly mention how our argument generalizes to words other than $\cat$.

\section{Proofs}\label{sec:higher}
We now prove Theorem~\ref{thm:multi}. Let $\omega=e^{2\pi i/n}$, and recall that the characters on the group $(\mathbb{Z}/n\mathbb{Z})^d$ are the functions
$$\chi_y(x):=\omega^{y \cdot x},$$
as $y$ ranges over the elements of $(\mathbb{Z}/n\mathbb{Z})^d$.  We will study the $n^d \times n^d$ matrix $M$ whose rows and columns are indexed by $(\mathbb{Z}/n\mathbb{Z})^d$; the $xy$-entry of $M$ is defined to equal $1$ if $x-y \in \{-1,0,1\}^d \setminus \{0\}$, and to equal $0$ otherwise.  Since $M$ is the adjacency matrix of a Cayley graph of $(\mathbb{Z}/n\mathbb{Z})^d$, the eigenvectors of $M$ are precisely the characters of $(\mathbb{Z}/n\mathbb{Z})^d$.  We can compute the eigenvalues explicitly: If $y$ has $b_k$ entries equal to $k$ for $k=0, \ldots, n-1$, then the eigenvalue $\lambda_y$ for $\chi_y$ is
\begin{align*}
\lambda_y=M\chi_y(0) &=\sum_{v \in \{-1,0,1\}^d \setminus \{0\}} \chi_y(v)\\
 &=-1+\prod_{k=0}^{n-1} \sum_{i=0}^{b_k} \sum_{j=0}^{b_k-i}\binom{b_k}{i}\binom{b_k-i}{j} \omega^{k(i-j)}\\
 &=-1+\prod_{k=0}^{n-1} \sum_{i=0}^{b_k} \binom{b_k}{i} \omega^{ki} \sum_{j=0}^{b_k-i}\binom{b_k-i}{j} \omega^{-kj}\\
 &=-1+\prod_{k=0}^{n-1} \sum_{i=0}^{b_k} \binom{b_k}{i} \omega^{ki} (1+\omega^{-k})^{b_k-i}\\
 &=-1+\prod_{k=0}^{n-1} (1+\omega^k+\omega^{-k})^{b_k}\\
 &=-1+\prod_{k=0}^{n-1}(1+2\cos(2\pi k/n))^{b_k}.
\end{align*}
In the second equality we split the coordinates of $y$ into level sets and then partitioned the $v$'s according to the number of $1$'s and $-1$'s in each such level set; the fourth and fifth equalities used the Binomial Theorem.  Notice that the largest eigenvalue is $\lambda_0=-1+3^d$.  All eigenvalues are greater than or equal to $-1-3^{d-1}$, and equality is achieved exactly when $y$ has one coordinate equal to $n/2$ and all other coordinates equal to $0$ (so necessarily $n$ is even for equality to hold).

We remark that this computation has a nice conceptual explanation.  Let $B$ be the adjacency matrix of the cycle graph on $n$ vertices.  Our matrix of interest can be expressed as
$$M=(B+\Id_n)^{\otimes d}-\Id_{n^d},$$
where $\Id_m$ denotes the $m \times m$ identity matrix and the superscript $\otimes d$ denotes the $d$-fold tensor product.  Hence the eigenvalues of $M$ are the quantities
$$(\mu_1+1) \cdots (\mu_d+1)-1,$$
as $\mu_1, \ldots, \mu_d$ range over eigenvalues of $B$ (and the eigenvectors of $M$ are tensor products of the eigenvectors of $B$).  The spectrum of $B$ is well known to be $\{2\cos(2\pi k/n): 0 \leq k \leq n-1\}$.  This behavior appears whenever one studies the adjacency spectrum of a graph product.

We are now ready to prove our main lemma, which can also be derived from known facts about max cuts in regular graphs with specified smallest eigenvalue (see, e.g., \cite{noga}, Lemma 3.1).  We define $\#(\A\C)$ to be the number of occurrences of $\A\C$ in our labeling, i.e., the number of pairs of points $x,x+y$ with the labels $\A$, $\C$ (in that order),  where $x \in (\mathbb{Z}/n\mathbb{Z})^d$ and $y \in \{-1,0,1\}^d \setminus \{0\}$.  We use analogous notation for counts of other configurations.

\begin{lemma}\label{lem:eigenval}
In any labeling of $(\mathbb{Z}/n\mathbb{Z})^d$ ($n \geq 3$), we have the inequality
$$\#(\A\C)+\#(\A\T)\leq 3^{d-1}n^d.$$
Moreover, equality is attained if and only if $n$ is even and the locus of the $\A$'s is of the form $\{(x_1, \ldots, x_d): x_r \equiv \epsilon \pmod{2}\}$ for some $1 \leq r \leq d$ and some $\epsilon \in \{0,1\}$.
\end{lemma}
\begin{proof}
Consider the function $f:(\mathbb{Z}/n\mathbb{Z})^d \to \{-1,1\}$ that equals $1$ at points labeled $\A$ and equals $-1$ at points labeled $\C$ and $\T$.  Then, with $M$ as defined above, we have
\begin{align*}
f^T M f &=2(\#(\A\A)+\#(\C\C)+\#(\T\T)+\#(\C\T)-\#(\A\C)-\#(\A\T))\\
 &=(3^d-1)n^d-4(\#(\A\C)+\#(\A\T)),
\end{align*}
so
\begin{equation}\label{eq:quadratic-form}
\#(\A\C)+\#(\A\T)=((3^d-1)n^d-f^T M f)/4.
\end{equation}
Expand $f$ as linear combination of characters, namely,
$$f=\sum_y a_y \chi_y.$$
Since $|f|=1$ everywhere, we have
$$n^d=\Vert f \Vert_2^2=\sum_y |a_y|^2 \cdot \Vert \chi_y \Vert_2^2=n^d \sum_y |a_y|^2$$
and hence
$$\sum_y |a_y|^2=1.$$
Now, again using the orthogonality of characters, we can bound
\begin{align*}
f^TMf &=\sum_y |a_y|^2 \cdot \Vert \chi_y \Vert_2^2 \cdot \lambda_y\\
 &\geq (-1-3^{d-1})n^d,
\end{align*}
with equality if and only if $n$ is even and the only nonzero $a_y$'s are the $y$'s with one entry equal to $n/2$ and all other entries equal to $0$.  We claim that since $f$ is $\{-1,1\}$-valued, this equality condition in fact implies that only a single such $a_y$ is nonzero.  Indeed, suppose the nonzero $a_y$'s correspond to $y_1, \ldots, y_k$, where each $y_i$ has $z_i$-entry equal to $n/2$ and all other entries equal to $0$.  Then, by considering the vectors $x$ with $z_i$-entries in $\{0,1\}$ (for $1 \leq i \leq k$) and all other entries equal to $0$, we see that $f(x)$ assumes all of the values
$$\sum_{i=1}^k \epsilon_i a_{y_i}$$
for $\epsilon_i \in \{-1,1\}$.
In particular, there are at least $k+1$ such values, and so $k$ must equal $1$, as claimed.  Notice that these possibilities correspond precisely to the configurations described in the ``moreover'' statement of the lemma.  

To complete the proof, we return to Equation~\ref{eq:quadratic-form} and find that
$$
\#(\A\C)+\#(\A\T) \leq ((3^d-1)n^d-(-1-3^{d-1})n^d)/4=3^{d-1}n^d,$$
as desired.
\end{proof}

Theorem~\ref{thm:multi} now follows easily.
\begin{proof}[Proof of Theorem~\ref{thm:multi}]
Since each occurrence of $\cat$ contains an occurrence of $\C\A$ and an occurrence of $\A\T$, we immediately deduce from Lemma~\ref{lem:eigenval} that
$$\#(\cat) \leq (3^{d-1}/2)n^d.$$
It remains only to characterize when equality holds.  Suppose $\#(\cat)=(3^{d-1}/2)n^d$.  Then, by Lemma~\ref{lem:eigenval}, $n$ is even and the locus of the $\A$'s is of the form $\{(x_1, \ldots, x_d): x_r \equiv \epsilon \pmod{2}\}$ for some $1 \leq r \leq d$ and some $\epsilon \in \{0,1\}$.  After translating and permuting the coordinates, we may assume that this locus is the set of points $(x_1, \ldots, x_d)$ with $x_d$ even.

We now analyze the $\C$'s and $\T$'s.  If the point $(x_1, \ldots, x_d)$ is labeled $\C$, then $(x_1, \ldots, x_d+2)$ must be labeled $\T$ and hence all of
$$(x_1 \pm 2, x_2, \ldots, x_d), (x_1, x_2 \pm 2, x_3, \ldots, x_d), \ldots, (x_1, \ldots, x_{d-2}, x_{d-1}\pm 2, x_d)$$
are labeled $\C$.  The same holds with the roles of $\C$ and $\T$ exchanged.  Iterating this observation and conditioning on the labels of the points $\{0,1\}^{d-1} \times \{1\}$, we arrive at the generalized Patchell--Spiro labelings.  (Notice that $n$ must be a multiple of $4$ in order for the $\C$'s and $\T$'s to alternate as we increase $x_d$ by increments of $2$.)  This completes the proof.
\end{proof}

\section{Lions, tigers, and other felines}\label{sec:conclusion}
Patchell and Spiro also asked what happens if $\cat$ is replaced by another word.  In particular, they asked about the case of words with all distinct letters, such as $\mathsf{LION}$ and $\mathsf{TIGER}$.  Our argument from the previous section gives the correct count for the maximum number of occurrences of such a word and completely confirms Conjecture 1 of~\cite{sam}.  (We thank Sam Spiro for bringing this to our attention.)

\begin{theorem}
Let $w$ be a word of length $r$ ($r \geq 2$) in which all letters are distinct.  In any labeling of $(\mathbb{Z}/n\mathbb{Z})^d$ ($n \geq r$), the number of occurrences of $w$ is at most $(3^{d-1}/(r-1))n^d$.  Moreover, equality is attained if and only $n$ is a multiple of $2r-2$ and the labeling is (the obvious $w$-analogue of) a generalized Patchell--Spiro labeling.
\end{theorem}

Let us sketch the proof.  Let $\mathcal{O}$ denote the set of letters appearing at odd indices of $w$, and let $\mathcal{E}$ denote the set of letters appearing at even indices of $w$.  Define the function $f:(\mathbb{Z}/n\mathbb{Z})^d$ to equal $1$ at points labeled with elements of $\mathcal{O}$ and to equal $-1$ at points labeled with elements of $\mathcal{E}$.  The argument of Lemma~\ref{lem:eigenval} shows that
$$\#(\mathcal{OE}) \leq 3^{d-1}n^d$$
(where $\#(\mathcal{OE})$ counts occurrences of $\mathsf{OE}$ for $\mathsf{O} \in \mathcal{O}$ and $\mathsf{E} \in \mathcal{E}$), and we have the analogous characterization of when equality occurs.  Moreover, $$\#(w) \leq \#(\mathcal{OE})/(r-1)$$ (using the fact that each $\mathcal{OE}$ can appear in at most one $w$), and so
$$\#(w) \leq (3^{d-1}/(r-1))n^d.$$ The characterization of equality goes just as in the proof of Theorem~\ref{thm:multi}.  It is easy to see that the generalized Patchell--Spiro labelings have $(3^{d-1}/(r-1))n^d$ occurrences of $w$.

In fact, this argument also works for some words with repeated letters.  In particular, we can handle any word in which: (1) no letter appears at both even and odd indices; and (2) no pair of letters appears consecutively more than once (counting both forward and backward occurrences).  An example of such a word is $\mathsf{FELINE}$, and a non-example is $\mathsf{ELEPHANT}$ (because the pair $\mathsf{EL}$ appears twice).  Even for some examples violating (2), it is possible to obtain a tight bound along the lines described above; we leave the details to the interested reader.

\section*{Acknowledgments}
We thank Sam Spiro for helpful discussions. The first author is supported in part by
NSF grant DMS--2154082.  The second author is supported in part  by the NSF Graduate Research Fellowship Program under grant DGE--2039656.

\end{document}